\documentclass{amsart}
\usepackage{bbm}
\usepackage{lipsum}
\usepackage{stmaryrd}
\usepackage{enumerate}  
 \usepackage[all]{xy}
\usepackage{showlabels}
\usepackage{amsmath,amsthm,amsfonts,amssymb,amscd,mathrsfs}
\usepackage[pdftex, colorlinks=true]{hyperref}
\newtheorem{theo}{Theorem}[section]
\newtheorem{prop}{Proposition}[section]
\newtheorem{lemma}{Lemma}[section]
\newtheorem{cor}{Corollary}[section]
\newtheorem{defi}{Definition}[section]
\newtheorem{ques}{Question}[section]
\newtheorem{rem}{Remark}[section]

\def \htop{\mathsf{h_{top}}}
\def \hbor{\mathsf{h_{bor}}}
\def \ocap{\mathsf{ocap}}

\def \Id {\mathsf{Id}} 
\begin{document}
\begin{large}

 \title[Topological and almost Borel Universality]{Topological and almost Borel Universality for systems with the weak specification property}
\author{David Burguet}
\email{david.burguet@upmc.fr}
\maketitle

\begin{abstract}We show that systems with some specification properties are topologically or almost Borel universal, in the sense that any aperiodic subshift with lower entropy may be topologically or almost Borel  embedded. 
\end{abstract}

\section{Introduction}

Usually a universal system is defined with respect to a collection $\mathcal{C}$ of measure preserving ergodic aperiodic transformations : it is a topological system\footnote{By a topological system $(Y,S)$ we will always mean that $Y$ is compact metric space and $S$ is a homeomorphism  from $Y$ into itself with finite topological entropy. The distance on $Y$ will be denoted by $d$ or $d_Y$.} $(Y,S)$ such that for any system $(X,T,\mathcal{A},\nu)$ of $\mathcal{C}$ there is a Borel $S$-invariant ergodic probability measure $\mu$ such that  $(Y,S,\mathcal{B},\mu)$ with $\mathcal{B}$ being  the Borel $\sigma$-algebra, is isomorphic to $(X,T,\mathcal{A},\nu)$. We will refer to such systems as \textbf{measure theoretical $\mathcal{C}$-universal} systems. The system is called \textbf{fully $\mathcal{C}$-universal systems} when we can assume the measure $\nu$ has full support. If one requires also that for any Borel probability $S$-invariant measure $\mu$ the system $(Y,S,\mathcal{B},\mu)$ is isomorphic to some  element of $\mathcal{C}$, then we will say that $(Y,S)$ is strictly $\mathcal{C}$-universal.

We focus here  on the collection $\{h\in I\}=\{(X,T,\mathcal{A},\nu), \ h(\nu,T)\in I\}$ for some interval $I$ of $\mathbb{R}^+$. Krieger's generator theorem \cite{Kr} claims the full shift with $K$-symbols is $\{h< \log K\}$-universal. More recently T. Downarowicz and J. Serafin \cite{DS} have built strictly $\{h\in I\}$-universal systems for any non-degenerated intervals $I$. It was also previously established by J. Serafin \cite{zero} that there is no $\{h=0\}$-strictly universal system. Concerning fully universal systems, A.Quas and T.Soo proved by using Baire arguments on joinings that some systems $(Y,S)$ with a specification  property are fully $\{h<h_{top}(S)\}$-universal.  Theorem \ref{fin} and Theorem \ref{lat} below generalize their result. D.A.Lind and J.-P.Thouvenot \cite{thou} proved earlier that hyperbolic (two-dimensional) toral automorphisms are fully universal. \\

One can also consider universal problems in the  topological and Borel  setting. A topological system $(Y,S)$ is aperiodic when we have $S^nx\neq x$ for all $x\in Y$ and for all positive integers $n$. When $Y$ is a zero-dimensional set, the topological system $(Y,S)$ is said to be zero-dimensional. Such a system is expansive when there exists a partition $P$ of $Y$ in clopen sets such that the maximum of the diameters of atoms in $\bigvee_{k=-n}^nS^{-k}P$ is decreasing to zero when $n$ goes to infinity. Let us now consider a collection $\mathcal{C}$ of expansive aperiodic  zero-dimensional systems (e.a.z. systems for short). We say a topological system $(Y,S)$ is \textbf{$\mathcal{C}$-universal} when for any $(X,T)$ in $\mathcal{C}$ there is a subsystem $(Z,R)$  of 
$(Y,S)$ topologically conjugated to $(X,T)$ ; in other words the system $(X,T)$ may be topologically embedded into $(Y,S)$. Here  we mainly  focus on the 
collection $\{\htop\in I\}=\{(X,T)\ \text{ e.a.z. with }\htop(T)\in I\}$ for an interval $I$. Around ten years after the generator theorem, Krieger proved the full 
shift with $K$-symbols is $\{ \htop< \log K \}$-universal \cite{Kri}. The proof uses 
similar tools (Rohklin towers) but in some sense the last theorem is  easier to 
prove because with the  expansiveness property we get a generator in "one 
step" whereas a multiscale limit approach is necessary in the generator theorem. \\

By a Borel system $(X,T)$ we mean a bijection  $T$ of a standard Borel space $X$ such that $T$ and $T^{-1}$ are invertible. The Borel entropy $\hbor(T)$ is  the supremum of the entropies $h(\nu,T)$ taken over all $T$-invariant measures $\nu$ (according to the variational principle $\hbor(T)=\htop(T)$ when $(X,T)$ is a topological system).  A Borel system $(X,T)$ \textbf{almost Borel embeds}  in another one $(Y,S)$ when there is a Borel injective map $\psi:X'\rightarrow Y$ satisfying $\psi\circ T=S\circ \psi$ with $X'$  being a  full subset of $X$, i.e. a subset of full measure for any $T$-invariant  probability measure. 
When  every Borel system $(X,T)$ in a collection
$\mathcal{C}$ of Borel aperiodic systems almost Borel embeds  in $(Y,S)$ we will say $(Y,S)$ is \textbf{almost Borel  $\mathcal{C}$-universal}. In particular when any system in $\mathcal{C}$ is a topological system,  a $\mathcal{C}$-universal system is almost Borel $\mathcal{C}$-universal.  In \cite{Hoc} M.Hochman showed that a subshift of finite type $(Y,S)$ is $\{\hbor< \htop(S)\}$-universal with $\{\hbor< \htop(S)\}=\{(X,T)\ \text{ aperiodic Borel with }\hbor(T)<\htop(S)\}$.  \\

In this paper we investigate the universality property of systems with the specification property (see \cite{kw} for a panorama on systems with specification-like properties). For $\epsilon>0$ a system $(Y,S)$ is said to have the \textbf{almost $\epsilon$-weak specification property} when for any pieces of orbits
\begin{center}
$T^{a_1}x_1,T^{a_1+1}x_1,...,T^{b_1}x_1$, ...., $T^{a_p}x_p,T^{a_p+1}x_1,...,T^{b_p}x_p$ 
\end{center} with $a_1<b_1<a_2<...<b_p$ there is a point $x\in Y $ with $d(T^kx, T^kx_i)<\epsilon$ for $k\in \bigcup_i[a_i,b_i]$ provided that $b_{i+1}-a_i\geq  L(b_{i+1}-a_{i+1})$ for some function $L=L_\epsilon$ satisfying $\lim_{n \rightarrow +\infty}L(n)/n=0$. When $L$ is a constant function then $(Y,S)$ is said to satisfy the \textbf{$\epsilon$-weak specification property}. 

The system $(Y,S)$ has the \textbf{(resp. almost) weak specification property}  when it satisfies the (resp. almost) $\epsilon$-weak specification property for all $\epsilon>0$.   It is easily seen that a subshift satisfies the (almost) weak specification property whenever it satisfies the (almost) $ \epsilon$-weak specification property for some $\epsilon>0$. \\

Our main result related to topological universality follows:
\begin{theo}\label{top}
Any subshift $(Y,S)$ with the weak specification property is $\{\htop<\htop(S)\}$-universal.\end{theo}

In fact we prove this universality property for any subshift with the almost coded weak  specification property (see Definition \ref{cco}). We do not know any example of subshift with the almost weak specification property which does not satisfy  the almost coded weak specification property.\\

For general  systems with the almost weak specification property, we establish the almost Borel universality. 

\begin{theo} \label{fin}
Any topological system $(Y,S)$ with the almost weak specification property is almost Borel $\{\hbor<\htop(S)\}$-universal.
\end{theo}

By Jewett-Krieger theorem, for any measure preserving aperiodic  ergodic transformation there is a uniquely ergodic (aperiodic) subshift whose unique invariant measure is isomorphic to this given measure preserving system. Thus any topological system $(Y,S)$ which is $\{\htop<\htop(S)\}$-universal or   almost Borel $\{\hbor<\htop(S)\}$-universal is also measure theoretical $\{h<\htop(S)\}$-universal. As a corollary we give an  affirmative answer to  a question raised  by A.Quas and T.Soo.

\begin{cor}\label{co}
Any topological system $(Y,S)$ with the almost specification property is measure theoretical $\{h<\htop(S)\}$-universal.
\end{cor}
 
 It was previously proved by A.Quas and T.Soo in \cite{trans} for  systems satisfying  also the asymptotic $h$-expansiveness and the small boundary property. Thus these two additional hypothesis are useless. 
 
 A topological system $(X,T)$ \textbf{almost Borel embeds finitarily} in another one $(Y,S)$ when there is a Borel map $\psi:X\rightarrow Y$ with $\psi\circ T=S\circ \psi$ such that $\psi$ is a 
finitary embedding, i.e. $\psi$ is continuous and injective on  a full set, i.e. a set of full measure for any $T$-invariant measure. 

\begin{ques}[See also Question 1 in \cite{mod}]
Is any topological system $(Y,S)$ with the almost specification property almost Borel  finitarily $\{h<\htop(S)\}$-universal, in the sense that every e.a.z. system $(X,T)$ in $\{h<\htop(S)\}$ almost Borel embeds finitarily in $(Y,S)$?
 \end{ques}

We give now an overview of the main ideas of the proof. To a given system $(Y,S)$ with the (almost) weak specification property we associate  intermediate subshifts of finite type, called specifications. The topological entropy of specifications may be arbitrarily close to $\htop(S)$  so that it is enough to embed these particular subshifts in the system by the universality of subshifts of finite type.  Under some conditions the specification may be embedded in the system via a closed set valued upper semicontinuous map. For subshifts with the weak specification one can derive a (single valued) continuous embedding proving Theorem \ref{top}. For systems with the almost weak specification property (proof of  Theorem \ref{fin}) we need to develop a multiscale limiting approach (we use the specification property at smaller and smaller scales $\epsilon$) to build an almost Borel embedding.

The present paper is organized as follows. The main tools (in particular the admissible specifications) are introduced  in Section 2, which ends with the proof of Theorem 1. The section 3 is devoted to the proof of Theorem 2  and we discuss in the last section the question of  fully almost Borel universal systems.

\section{Topological embedding  in systems with the specification property}

\subsection{Intermediate subshifts of finite type} We let $\mathbb{N}^*$ be the set of positive integers. A  nondecreasing function  $L:\mathbb{N}\rightarrow \mathbb{N}^*$ with $\lim_{n\rightarrow +\infty}\frac{L(n)}{n}=0$ is said to be \textbf{sublinear}. Given a topological system $(Y,S)$ and a sublinear function $L$ we call a \textbf{$L$-specification} of $(Y,S)$ any subshift of finite type $(\Lambda,\sigma)$ as follows. 

There exist  positive integers $m,l$, a  finite set $E\subset Y$, an integer valued function $N:E\rightarrow \mathbb{N}$ on $E$ with $l\geq L\left(\max_{x\in E}N(x)\right)$, a distinguished point 
$o\in Y$, called the \textbf{marker point }of $\Lambda$,  and an added symbol $*\notin Y$, such that $\Lambda\subset \left(Y\cup \{*\}\right)^\mathbb{Z}$ is the subshift of finite type  generated by the family of  words $(w_x)_{x\in E}$ where $w_x$ is defined for all $x \in E$ by $$w_x:=(o,So,...,S^mo,*^{L(m)},x,Sx,...,S^{N(x)}x, *^l).$$ The word of length $n\in \mathbb{N}^*$ given by $n$ consecutive symbols $*$ is here denoted  by $*^n$. In other terms the elements of $\Lambda$ are given by  the infinite concatenations of such words. The special symbol $*$ labels the places used to \textit{glue} the finite orbits of $x$ and $o$ via the specification property. When  the lengths of the generating words $(w_x)_{x\in E}$ are multiplicatively independent, the $L$-specification of $(Y,S)$ is a  topologically mixing  subshift of finite type.  
We recall that the shift map $\sigma:\Lambda\rightarrow \Lambda$ is defined by 
$\sigma((u_n)_n)=(u_{n+1})_n$ for all $(u_n)_n\in \Lambda$.

We let  $\underline{N}:=\min_{x\in E}N(x)$ and
 $\overline{N}:=\max_{x\in E}N(x)$. The topological entropy of a $L$-specification is mostly given by the cardinality of the set $E$ when  $m/\overline{N}$ and $l/\overline{N}$ are small enough. More precisely we have:

\begin{lemma}\label{entr} For any sublinear  function  $L$  and for any $\gamma>0$ there is $\delta>0$ so that 
 any $L$-specification  with datas $(E,o,N,m,l)$ satisfying  $\min\left(m/\overline{N},l/\overline{N}\right)<\delta$  has topological entropy in $[\frac{\log \sharp E}{\overline{N}}(1-\gamma),\frac{\log \sharp E}{\underline{N}}]$. 
\end{lemma}
\begin{proof}
The topological entropy of the subshift $(\Lambda,\sigma)$ is given by the exponential growth in $n$ of the cardinality of the set $\mathcal{L}_n(\Lambda)$ of words of $\Lambda$ with length less than or equal to $n$. We let $q=m+L(m)+l$ so that the length $|w_x|$ of $w_x$ for $x\in E$ is equal to $N(x)+q$.  
For any positive integer $P$ the map from $E^P$ to  $\mathcal{L}_{P(\overline{N}+q)}(\Lambda)$ sending $(x_0,...,x_{P-1})$ to  the concatenation $w_{x_0}w_{x_1}\cdots w_{x_{P-1}}$ is injective so that 
\begin{align*}
 \htop(\Lambda)&=\lim_n\frac{1}{n}\log \sharp \mathcal{L}_n(\Lambda),\\
 &\geq \lim_P\frac{1}{P(\overline{N}+q)}\log \sharp E^P=\frac{1}{q+\overline{N}} \log \sharp E.
\end{align*}
We have $q\leq m+L(\overline{N})+l\leq m+2l$ so that for any $\gamma>0$ we get for $m/\overline{N}$ and $l/\overline{N}$ small enough 
$$\htop(\Lambda)\geq \frac{1}{q+\overline{N}} \log \sharp E\geq \frac{\log \sharp E}{\overline{N}}(1-\gamma).$$

In the other hand any word in $\mathcal{L}_{n\underline{N}}(\Lambda)$ is a subword of a concatenation of $n+1$-generating words. Therefore we have 
$\sharp \mathcal{L}_{n\underline{N}}(\Lambda)\leq \sup_{x\in E}|w_x|\sharp E^{n+1}$
and finally we conclude 
\begin{align*}
 \htop(\Lambda)&=\lim_n\frac{1}{n\underline{N}}\log \sharp \mathcal{L}_{n\underline{N}}(\Lambda)\leq \frac{\log \sharp E}{\underline{N}}.
\end{align*}
We note that this last inequality holds true for any $N$. 
\end{proof}

 By Krieger's theorem, which we recall below, we may embed topologically any e.a.z. system  in a toplogically mixing specification with larger entropy.

\begin{lemma}[Krieger \cite{Kri}]\label{kk}
Let $(X,T)$ be an e.a.z. system. Then $(X,T)$ embeds topologically in  any topologically  mixing  subshift of finite type $(Y,S)$ with $\htop(S)>\htop(T)$. In other terms any topologically mixing subshift of finite type is $\{\htop<\htop(S)\}$-universal.
\end{lemma}

We also recall the corresponding theorem due to M.Hochman for  the almost Borel universality.

\begin{lemma}[Krieger \cite{Hoc}]\label{kki}
Let $(X,T)$ be an aperiodic Borel system. Then $(X,T)$almost Borel  embeds  in  any topologically  mixing  subshift of finite type $(Y,S)$ with $\htop(S)>\hbor(T)$. In other terms any topologically mixing subshift of finite type is $\{\hbor<\htop(S)\}$-universal.
\end{lemma}

\subsection{Closed set valued upper semicontinuous embeddings of admissible specifications}
Let $\epsilon>0$ and let $L$ be a sublinear function. For any $y\in Y$ and $\epsilon>0$ we denote the closed ball at $y$ of radius $\epsilon$ by $B(y,\epsilon)$.
The system $(Y,S)$ is said to have the \textbf{$(\epsilon,L)$-specification  property} when for any  $L$-specification $(\Lambda,\sigma)$ and any $x=(x_k)_k\in \Lambda$ the set   
$$\Delta_\epsilon^\Lambda(x):=\bigcap_{k, \  x_k\neq *}S^{-k}\left(B( x_k, \epsilon)\right) \text{ is non empty.}$$

The map $\Delta_\epsilon^\Lambda$ is a closed set valued equivariant map from $(\Lambda,\sigma)$ to $(\mathcal{K}(Y), S)$, where $\mathcal{K}(Y)$ is the set of closed subsets of $Y$ endowed with the Hausdorff distance. We recall that the  Hausdorff distance between $F,K\in \mathcal{K}(Y)$ is given by the infimum of the positive real numbers $r$ such that $F$ and $K$ are respectively  contained in the $r$-neighborhoods of $K$ and $F$. 
For a finite cylinder $w$ in $\Lambda$ we let $\Delta_\epsilon^\Lambda(w):=\bigcup_{y\in w\subset \Lambda}\Delta_\epsilon^\Lambda(y)$. It is also convenient for the purpose of the next subsections to define the  set $\Delta_\epsilon(v)$ for a general finite word $v$ with letters in $Y\cup\{*\}$  as $$\Delta_\epsilon(v):=\bigcap_{m\leq k\leq n, \  v_k\neq *}S^{-k}\left(B( v_k, \epsilon)\right)$$ with $v=v_m...v_n$ and  $v_i\in Y\cup\{*\}$ for $i$ in the interval of integers $\llbracket m,n\rrbracket$ (called the interval of coordinates of $v$). Observe that when $v$ is a word of $\Lambda$ and $[v]$ is the associated cylinder given by $[v]=\{(w_k)_{k\in \mathbb{Z}}\in \Lambda, \ w_i=v_i \text{ for } i\in \llbracket m,n\rrbracket\}$, we have $\Delta_\epsilon^\Lambda([v])\subset \Delta_\epsilon(v)$ but these  two sets may a priori differ. \\

The closed set valued map $\Delta_\epsilon^\Lambda$ is said to be injective when  distinct points have disjoint images.  We consider now a topological system $(Y,S)$ with the $(\epsilon, L)$-specification property and we give sufficient conditions on the specification $\Lambda$ to ensure the injectivity of $\Delta_\epsilon^\Lambda$. For $y\in Y$, $n\in \mathbb{N}$ and $\epsilon>0$ we denote by $B(y,n,\epsilon)$ the $n$-dynamical ball  at $y$ of size $\epsilon$ defined by 
 $$B(y,n,\epsilon):=\bigcap_{0\leq k<n} S^{-k}B(S^ky,\epsilon).$$
 We will also let $r(n,\epsilon, S)$ be the minimal number of $n$-dynamical balls   of size $\epsilon$ which are covering $Y$. We recall the topological entropy of $Y$ is given by
 $$\htop(Y,S):=\lim_{\epsilon\rightarrow 0}\limsup_n\frac{1}{n}\log r(n,\epsilon,S).$$

\begin{defi}\label{adm}
A $L$-specification is said  $\epsilon$-admissible when the associated datas $(E,o,N,m,l)$ satisfies :
\begin{enumerate}
\item $12l<m<\underline{N}$,
\item the set $E$ is a $(N,\epsilon)-$separated, i.e. for all $x\neq y\in E$ we have either $N(x)\neq N(y)$ or $B(x,N(x),\epsilon)\cap B(y,N(y),\epsilon)=\emptyset$,
\item $B(o,m,\epsilon)\cap S^kB(o,m,\epsilon)=\emptyset$ for any $0< k\leq \frac{3}{4}m $,
\item $B(o,m,\epsilon)\cap S^kB(x,N(x),\epsilon)=\emptyset$ for any $-\frac{2}{3}m\leq k\leq N(x)-\frac{2}{3}m$ and for any $x\in E$. 
\end{enumerate}
\end{defi}
According to the fourth item the marker point $o$ does not belong to $E$ when the specification is $\epsilon$-admissible. 
In the next lemma, we show that the map $\Delta_\epsilon^{\Lambda}$ is injective for an $\epsilon$-admissible $L$-specification $\Lambda$. The conditions 3 and 4 in the above definition are respectively used to locate the concatened words $w_x$ in  an element of $\Lambda$ whereas the condition 2 allows us to identify the points $x$. 

\begin{lemma}\label{ent}
Assume a $L$-specification  $\Lambda$ is $\epsilon$-admissible.  Let  $A_1B_1$ (resp. $A_2B_2$)  be the concatenation of $A_1$ and $B_1$ (resp. $A_2$ and $B_2$) with $A_i,B_i$ being generating words of $\Lambda$. We assume moreover the intervals of coordinates $\llbracket m_1,n_1\rrbracket$ and $\llbracket m_2,n_2\rrbracket$  of $A_1$ and $A_2$ are not disjoint. Then 
$$\left[\Delta_\epsilon(A_1B_1)\cap \Delta_\epsilon(A_2B_2)\neq \emptyset\right]\Rightarrow \left[A_1=A_2\right].$$

In particular  the closed set valued map $\Delta_\epsilon^\Lambda$ is injective. 
\end{lemma}

\begin{proof}
Let $l\in  \llbracket m_1,n_1\rrbracket \cap \llbracket m_2,n_2\rrbracket$ and $u\in \Delta_\epsilon(A_1B_1)\cap \Delta_\epsilon(A_2B_2)$. We first prove $m_1=m_2$. 
We argue by contradiction : we may assume $m_1<m_2$ by swapping $m_1$ and $m_2$ if necessary. We have the following possibilities: 
\begin{itemize}
\item  $m_1<m_2\leq m_1+3m/4$ or $(n_1-3m/4<) \ n_1-\left(2m/3+l\right) \leq m_2\leq n_1$, then in both cases we have $S^{m_2}u\in B(o,m,\epsilon)\cap   S^kB(o,m,\epsilon)$ for some $k$ with $|k|\leq 3m/4$ contradicting item (3), 
\item  $m_1+3m/4<m_2<n_1-\left(2m/3+l\right)$, then if we let $w_x$ be the generating word given by $A_1$ we get $S^{m_2}u\in  B(o,m,\epsilon)\cap  S^kB(x,N(x),\epsilon)$ for some $k$ with $-2m/3\leq k\leq N(x)-2m/3$ because we have  $m/4+L(m)$ ($<2m/3$). This contradicts item (4).
\end{itemize}

In a similar way we prove now that $n_1=n_2$ (however note that the right boundaries of the intervals of coordinates of   $B_1$ and $B_2$ may differ). In the above proof of  $m_1=m_2$ we just have  used $u\in \Delta_\epsilon(A_1)\cap \Delta_\epsilon(A_2)$ and the fact that the intervals of coordinates of $A_1$ and $A_2$ were not disjoint. Therefore to show $n_1=n_2$ we only need to check that the  intervals of coordinates of $B_1$ and $B_2$ are not disjoint.  Let $w_{x'}$ and $w_y$ be respectively the generating words given by $B_1$ and $A_2$. We argue by contradiction. Without loss of generality we assume that the intervals of coordinates of $A_1$ and $B_1$ are contained in the interval of coordinates of $A_2$, which implies $n_2-n_1\geq \underline{N}+l+m+L(m)$ and thus 
\begin{align*}
N(y)-(n_1-m_1)&=N(y)+m_2-n_1,\\
&= n_2-m_2-(l+m+L(m))+m_2-n_1,\\
&\geq \underline{N}>m.
\end{align*}
From $\Delta_\epsilon(A_1B_1)\cap \Delta_\epsilon(A_2B_2)\neq \emptyset$  we   have  $B(o,m,\epsilon)\cap S^{k}B(y,N(y),\epsilon)\neq \emptyset$ with $0<k=n_1-m_1-(m+L(m))<n_1-m_1<N(y)-m$ so that we get again a contradiction with item (4) in the Definition \ref{adm} of the $\epsilon$-admissibility.

Finally as the points in $E$ are $(N,\epsilon)$-separated by item (2), the sets $\Delta_\epsilon(A_1)$ and $\Delta_\epsilon(A_2)$ are disjoint if and only if $A_1$ and $A_2$ are distinct.
\end{proof}  

We need in our proof to assume $\Delta_\epsilon(A_1B_1)\cap \Delta_\epsilon(A_2B_2)\neq \emptyset $ and not only $\Delta_\epsilon(A_1)\cap \Delta_\epsilon(A_2)\neq \emptyset $ in order to identify the last coordinates of $A_1$ and $A_2$. This comes  from the fact the marker appears only at the beginning of a generating word.

\subsection{Construction of admissible specifications.}
When the integer valued function $N:E\rightarrow \mathbb{N}^*$ satisfies $N(E)=\{n_0,n_0+1\}$ for some $n_0$  and $l=L(n_0+1)$ we say $\Lambda$ is a \textbf{simple specification}.  

\begin{prop}\label{pp}
Let $0<\alpha<\htop(S)$. For small enough $\epsilon>0$  and for any sublinear function $L$ there is an $\epsilon$-admissible simple $L$-specification with topological entropy larger than $\alpha$.
\end{prop}

The proposition will follow from the two following lemmas. We first show the existence of a marker point $o$ satisfying the item (3) in Definition \ref{adm}.

\begin{lemma}\label{vvo}Let $(Y,S)$ be a topological system with $\htop(S)>0$. 
For $\epsilon>0$ small enough there is for large $m$ a point $o=o(m,\epsilon)\in Y$ with $S^kB(o,m,\epsilon)\cap B(o,m,\epsilon) =\emptyset$ for all $0< k\leq 3m/4$.
\end{lemma}
\begin{proof}
It is an easy consequence of Ornstein-Weiss return time formula, that  we recall now.  For any $\epsilon>0$ and $x\in Y$ we  let $$h^{OW}(x,\epsilon):=\limsup_n\frac{1}{n}\log \min\{k>0, \ S^kx\in B(x,n,\epsilon)\}$$ and then for any ergodic $S$-invariant measure $\mu$ : $$h^{OW}(\mu,\epsilon):=\int h^{OW}(x,\epsilon)d\mu(x).$$  Then Ornstein-Weiss return time \cite{dw} formula states that 
$$h(\mu)=\lim_{\epsilon\rightarrow 0} h^{OW}(\mu,\epsilon).$$
Let's go back to the proof of the lemma. Let $\mu$ be an ergodic measure with positive entropy and $\epsilon>0$ with $h^{OW}(\mu,2\epsilon)>0$. In particular there is $o\in Y $ with $h^{OW}(o,2\epsilon)=\limsup_n\frac{1}{n}\log \min\left\{k>0, \ S^ko\in B(o,n,2\epsilon)\right\}>0$. Then  there are arbitrarily  large integers $m$ such that for any $0\leq k\leq 3m/4$ we have  $S^ko \notin B(o,m/4,2\epsilon)$  and thus
$$S^kB(o,m,\epsilon)\cap B(o,m,\epsilon)\subset B(S^ko,m/4,\epsilon)\cap B(o,m/4,\epsilon) =\emptyset.$$
 This concludes the proof of the lemma. 
\end{proof}


We show now that we can find an  $(N,\epsilon)$-separated set  $E$ with large cardinality with $N(E)=\{n_0,n_0+1\}$ which satisfies property (4) in Definition \ref{adm} (observe that an $(n_0,\epsilon)$-separated set  $E$ is also $(N,\epsilon)$-separated in this case). 

\begin{lemma}\label{ds} Fix $1>\delta>0$ and  $\epsilon>0$. For $n$ large enough, for any $y\in Y$ and for  
 any $(n,6\epsilon)$-separated set  $E$  the subset $F$ of $E$ given by   elements $x \in E$ satisfying 
 $S^lx\in B(S^ky,\delta n/2,2\epsilon)$ for some $(k,l)\in \mathcal{E}$    has cardinality less than  $e^{\sqrt{1-\delta/2}\htop(S)n}$ with 
 $$\mathcal{E}:=\{(k',0), \ 0\leq k'\leq \delta n \}\cup\{ (0,l'), \ 0\leq l'  \leq (1-\delta)n\}.$$
\end{lemma}

\begin{proof}
Let $\alpha>1$ with $\alpha(1-\delta/2)<\sqrt{1-\delta/2}$.  There is a constant $C$ such that $r(n,\epsilon/2,S)\leq Ce^{ \alpha\htop(S)n}$ for all $n$. Then for any $(k,l)\in \mathcal E$  the set $ S^{-l}B(S^{k}y,\delta n/2, 2\epsilon)$ may be covered by a family of $n$-dynamical balls of size $3\epsilon$ with cardinality   less than $C^2 e^{\gamma (1-\delta/2) \htop(S)n}$. Thus the union of all these sets over $(k,l)\in \mathcal{E}$ may be covered by a family of such dynamical balls  with cardinality less than $(1-\delta)\delta n^2C^2 e^{\alpha (1-\delta/2) \htop(S)n}<e^{\htop(S)\sqrt{1-\delta/2}n}$ for $n$ large enough. This concludes the proof as these dynamical balls contain at most one point of $E$.
\end{proof}

We prove now Proposition \ref{pp}. 
\begin{proof}[Proof of Proposition \ref{pp}]
For any $2/3>\delta>0$ we may take $n_0$ so large that $L(n_0+1)/n_0< \delta/8$. 
We let $m$ be the integer part of $3\delta n_0/2$. We take $\epsilon>0$ small enough and  $o$  the marker point of $Y$ given by Lemma \ref{vvo}. Then for any $(n_0,6\epsilon)$-separated set $G$
we consider the subset $F$ of $G$ given by Lemma \ref{ds} and we put $E=G\setminus F$.  We let $N(z)=n_0+1$ for some fixed $z\in E$ and $N(z')=n_0$ for $z'\neq z$ in $E$. We claim the $L$-specification associated to $(E,y,N,m)$ is $\epsilon$-admissible. Let us just check that the item (4) in Definition \ref{adm} is fullfilled. Arguing by contradiction there is  $x\in E$  with $B(o,m,\epsilon)\cap S^kB(x,N(x),\epsilon)\neq\emptyset$ for some $-\frac{2}{3}m\leq k\leq N(x)-\frac{2}{3}m$. Then we distinguish two cases :
\begin{itemize}
\item either $-\frac{2}{3}m\leq k\leq 0$, then we have  $\emptyset \neq B(x,N(x),\epsilon)\cap S^{-k}B(o,m,\epsilon)\subset B(x,m/3,\epsilon)\cap B(S^{-k}o,m/3,\epsilon)$ and thus $x\in B(S^{-k}o,\delta n_0/2,2\epsilon)$ with $0\leq -k\leq \delta n_0$;
\item or $0\leq k \leq N(x)-\frac{2}{3}m$, then we have $\emptyset \neq B(o,m,\epsilon)\cap S^kB(x,N(x),\epsilon) \subset B(o,2m/3,\epsilon)\cap B(S^kx,2m/3,\epsilon))$ and thus $S^kx\in B(o,\delta n_0/2,2\epsilon)$ with $0\leq k\leq (1-\delta) n_0$.
\end{itemize}
In both cases $x $ should belong to $F$ contradicting our assumption. 

By taking $\epsilon$ and $\delta$ small and $n_0$ large enough we may find according to 
Lemma \ref{entr} such a specification with topological entropy arbitrarily close to $\htop(S)$. \end{proof}

A closed set valued map $\phi:X\rightarrow \mathcal{K}(Y)$ is said \textbf{upper semicontinuous} when for all $x\in X$ and for all neighborhoods $V_x\subset Y$ of $\phi(x)$ there is a neighborhood  $U_x\subset X$ of $x$ with $\phi(x')\subset V_x$ for all $x'\in U_x$. From Proposition \ref{pp} and Lemma \ref{ent} we get:

\begin{cor}
For any system $(Y,S)$ with the almost weak specification property and for any e.a.z. system $(X,T)$ there is an upper semicontinous closed set valued  embedding of $(X,T)$ in $(\mathcal{K}(Y),S)$.  
\end{cor}


\subsection{Selector of a specification, the case of subshifts}

When $(Y,S)$ has the $(\epsilon,L)$-specification property, a  map  $\phi:\Lambda\rightarrow Y$ is said to be a \textbf{$\epsilon$-selector of the specification $\Lambda$} if $\phi(x)\in \Delta^\Lambda_\epsilon(x)$ for all $x\in \Lambda$. By Kuratowski-Ryll-Nardzewski measurable selection theorem a measurable selector always exists, as the closed set value map $\Delta^\Lambda_\epsilon$ is upper semicontinuous. However we are here interested in equivariant selectors, i.e. selectors $\phi$ satisfying $\phi\circ \sigma=S\circ \phi$. Invoking the axiom of choice there always exists an equivariant selector, but we do not know if there always exists a measurable one. We investigate now the existence of continuous equivariant $\epsilon$-selector of the specification $\Lambda$ in $(Y,S)$   when $(Y,S)$ is a subshift. 

\begin{defi}\label{cco} A subshift $(Y,S)$ is said to have the coded almost weak $L$-specification property for a positive  sublinear function $L$ when one can associate to any finite word  $w$ a finite word $\tilde w$ of the form $\tilde{w}=u_wwv_w$ such that :
\begin{itemize}
\item $|u_w|=c$ and  $ |v_w|=L(w)-c$ for some positive integer $c\leq L(0)=\inf_nL(n)$ independent of $w$,
\item for any finite words $w_1, w_2,...,w_n$ of $(Y,S)$ the concatenation $\tilde{w_1}\tilde{w_2}... \tilde{w}_n$ is a word of $(Y,S)$. 
\end{itemize}
We also say that $(Y,S)$ has the coded almost weak specification property when $(Y,S)$ satisfies the coded almost weak $L$-specification property for some sublinear function $L$. Finally when the function $L$ may be chosen equal to a constant then we speak of coded weak specification.   
\end{defi}

We believe that any subshift with the almost weak specification property satisfies the above coded almost weak specification property, but this remains an open  question. For a subshift $(Y,S)$ and an associated specification $\Lambda$ we consider the closed set valued map $\Delta^\Lambda:\Lambda\rightarrow \mathcal{K}(Y)$ defined by 
$\Delta^\Lambda((x_k)_k)=\{(z_k)_k\in Y, \ z_k=x_k \text{ for }x_k\neq *\}$ for any $x\in \Lambda$. It corresponds to the map $\Delta^\Lambda_\epsilon$ for $\epsilon=1$ and a well chosen metric on $Y$.

\begin{lemma}\label{ee}
Let $(Y,S)$ be a subshift satisfying the coded almost weak $L$-specification property. Then any  simple $L$-specification $(\Lambda,\sigma)$ with  datas $(E,o,N,m,l)$ satisfying $L(n_0)=L(n_0+1)$ with $N(E)=\{n_0,n_0+1\}$ admits a continuous equivariant selector.
\end{lemma}

\begin{proof}
For $x\in E$  (resp. $o$ being the marker point) we let $v_x$ (resp. $v_o$) be the word given by the $N(x)$ first letters of $x$ (resp. $M$ first letters of $o$). With the notations of Definition \ref{cco}  the SFT $\Lambda$ embeds into $Y$  via the map  sending any generating word $w_x$, $x\in E$, to the word $\sigma^{-c}\left(\tilde{v}_o\tilde{v}_x\right)$ of $Y$. Indeed the SFT generated by these last words is a subshift of $Y$ according to the coded almost weak $L$ specification property. Moreover this map is equivariant as the words $w_x$, $x\in E$ and  $\sigma^{-c}\left(\tilde{v}_o\tilde{v}_x\right)$  have the same length. Finally this embedding  defines a continuous selector of $\Lambda$ as we 
have shifted appropriately the interval of coordinates.
\end{proof}

A subshift is said \textbf{synchronised} when there is a word $u$, called the\textbf{ synchronizer word}, such that for any words $w$ and $v$ with $uw$ and $vu$ being words then $vuw$ is also a word. Any subshift with the weak specification property is synchronized \cite{jun}. 

\begin{lemma}\label{dd}
Any subshift $(Y,S)$ with the weak specification property satisfies the  coded specification property.
\end{lemma}
\begin{proof}
 Let $u$ be a synchronizer
word in $Y$.  By the weak specification property there exist a positive integer  $l$ such 
that for any $w$ in $Y$ there are  two words $a_w$ and $b_w$  of length $l$ with $
 ua_wwb_wu$ being a word. The property of the synchronizer clearly implies 
that the map $w\mapsto \widetilde{w}:=ua_wwb_w$ satisfies the condition of coded  weak specification (with $c=|u|+l$ and $L=l$). 
\end{proof}

\begin{proof}[ Proof of Theorem \ref{top}]
We can now prove Theorem \ref{top}. Let $(Y,S)$ be a subshift with the weak specification property. As already observed it is enough  by Lemma \ref{kk} to embed topologically subshifts of finite type with topological entropy arbitrarily close to $\htop(S)$. We build in Proposition \ref{pp}  such subshifts, the admissible simple specifications  $(\Lambda,\sigma)$, for which there is an injective equivariant closed set valued map $\Delta^\Lambda:(\Lambda,\sigma)\rightarrow  (\mathcal{K}(Y),S)$. Finally by Lemma \ref{dd} and Lemma \ref{ee} this map has a continuous equivariant selector, which gives the desired embedding of $(\Lambda,\sigma) $ in $(Y,S)$.
\end{proof}

\section{Almost Borel embedding  subshift  in systems with the specification property}

In the previous section we used the specification property at a fixed given scale. However we were able to prove a topological embedding only for subshifts with the weak specification property.  Here we will use the specification property at any arbitrarily small scale to build an almost Borel embedding in a  given system $(Y,S)$ with the almost weak specification property.

\subsection{Inverse limits}
Let $(X_k, T_k)_{k\in \mathbb{N}}$  be a sequence of topological systems  with semiconjugacies $\pi_k:(X_{k+1}, T_{k+1})\twoheadrightarrow (X_k,T_k)$ for all $k$ (we use the usual notation $\twoheadrightarrow$ to denote a surjective map). The inverse limit $(\varprojlim X_k,T)$ is the topological system given by $$\varprojlim_kX_k:=\{ (x_k)_k\in \prod_kX_k, \ \ \pi_k(x_{k+1})=x_k \text{ for all }k\}$$ with $T$ acting as $T_k$ on the $k^{th}$-coordinate for  all $k$. \\

For any $k\in \mathbb{N}$ we let 
$\eta_k:(\{0,1,...,k+1\}^\mathbb{Z},\sigma)\rightarrow (\{0,1,...,k\}^\mathbb{Z},\sigma)$  be  the surjective equivariant map which replaces the letter $k+1$ with  the letter $k$, i.e. the $l^{th}$-term of the sequence $\eta_k((x_l)_l)$  is equal to $k$ if $x_l= k+1$ and to $x_l$ if not, or equivalently $\eta_k((x_l)_l)=\left(\min(x_l,k)\right)_l$. Fix a nondecreasing sequence $\underline{n}=(n_k)_{k\geq 1}$ of positive integers. Let $Z^0_{\underline{n}}$ be the sequence with all terms equal to zero. We define by induction on $k\geq 1$ the subshift $Z_{\underline{n}}^k$ of $\{0,1,...,k\}^\mathbb{Z}$
as the set of $x=(x_l)_l$ satisfying:
\begin{itemize}
\item $\eta_{k-1}(x)\in Z_{\underline{n}}^{k-1}$, 
\item for two consecutive integers    $p<q$ in $\{i\in \mathbb{Z}, \ x_i=k\}$ the cardinality of  $ \{ i\in \llbracket p,q-1\rrbracket \text{ with }x_i=k-1\}$ is either $n_k$ or $n_{k}+1$.
\end{itemize}

We let $(Z_{\underline{n}},\sigma)$ be the inverse limit $\varprojlim Z_{\underline{n}}^k$ with semiconjugacies given by $(\eta_k)_k$  endowed with the shift 
map on each coordinate. Clearly the topological entropy of $Z_{\underline{n}}$ is arbitrarily small for fast 
enough increasing sequence $\underline{n}$.

\begin{lemma}\label{dal}
Let $\underline{n}=(n_k)$ be a nondecreasing sequence of positive integers. For any aperiodic Borel system $(X,T)$ there is a  Borel map $\psi:X'\rightarrow Z_{\underline{n}}$ satisfying $\psi\circ T=\sigma\circ \psi$ with $X'$ being a full subset of $X$. 
\end{lemma}

\begin{proof}
For an aperiodic Borel system $(X,T)$ there is for any positive integer $n_1$ a Borel set $U_1\subset X$ such that the first return in $U_1$ takes value in $\{n_1,n_1+1\}$ and $X'=\bigcup_{k=0}^{n_1+1}T^kU_1$ is a full invariant set. Indeed, firstly any aperiodic Borel system almost Borel embeds on a full set into the shift on $\mathbb{N}^\mathbb{Z}$ \cite{W} and  finally we may use for this shift  a Borel  version of Alpern's towers lemma \cite{GW}. This defines a Borel equivariant map $\psi_1:X\rightarrow Z_{\underline{n}}^1$ by letting $\psi_1(x)=\left(1_{U_1}(T^kx)\right)_k$ for all $x\in X'$. 
Then we may consider the Borel system $(U_1, T_{U_1})$ given by the first return  map in $U_1$. There 
is now a Borel set $U_2\subset U_1$ such that the first return in $U_2$ takes value in $\{n_2,n_2+1\}$ and we 
define now a Borel equivariant map $\psi_2:X\rightarrow Z_{\underline{n}}^2$ with $\eta_2\circ 
\psi_2=\psi_1$ by letting $\psi_2(x)=\left((1_{U_1}+1_{U_2})(T^kx)\right)_k$ for all $x\in X$. By 
pursuing the process we build a sequence of Borel equivariant maps $(\psi_k)_k$ compatible with 
the factor maps $(\eta_k)_k$ of the inverse limit $Z_{\underline{n}}$ so that this sequence 
defines a Borel equivariant map from $(X,T)$ to $(Z_{\underline{n}}, \sigma)$. 
\end{proof}

The lemma below follows easily from the above lemma and Hochman's theorem (Lemma \ref{kki}). 

\begin{lemma}\label{fest} For any nondecreasing sequence $\underline{n}$ of positive integers and any  topologically mixing  subshift of finite type $(Y,S)$ the product system $(Y\times Z_{\underline{n}},S\times \sigma)$ is almost Borel $\{\hbor<\htop(S)\}$-universal. 
\end{lemma}

Thus it is  enough to almost Borel embed in $(Y,S)$ such a product system $(X\times Z_{\underline{n}},T\times \sigma)$ with $\htop(T)$ arbitrarily close to $\htop(S)$.  Observe we may also write 
$X\times Z_{\underline{n}}$ as the inverse limit of $(X\times Z_{\underline{n}}^k)_k$. 
In fact we will build in the next subsection an inverse limit of specifications $(\Lambda_k, \sigma)_k$ almost Borel embedable in $(Y,S)$ such that $\varprojlim_k\Lambda_k$ is topologically conjugated to $\Lambda_0\times Z_{\underline{n}}$.

\begin{rem}
By using Krieger's marker lemma, see Lemma 8.5.4 in \cite{Dowb} one proves similarly for an a.e.z. system the existence of  a continuous map $\psi:Y\rightarrow Z_{\underline{n}}$ satisfying $\psi\circ S=\sigma\circ \psi$. Combined with Krieger's theorem  (Lemma \ref{kk})  we get that for any  topologically mixing  subshift of finite type $(Y,S)$ the product system $(Y\times Z_{\underline{n}},S\times \sigma)$ is  $\{\htop<\htop(S)\}$-universal. 
\end{rem}

\subsection{Almost Borel embedding of admissible inverse limit of specifications.}
\subsubsection{Nature of the almost Borel  embedding}\label{dff}
  To any specification $(\Lambda, \sigma)$ of $(Y,S)$ given by a marker point $o$ and a separated set $E$ we may associate  the function $\psi_\Lambda:\Lambda\rightarrow Y$  which maps $z=(z_k)_k\in \Lambda$ to $z_0$ whenever $z_0= S^kx$ for some $x\in E$ and maps $z$ to $o$ otherwise. Note this function is clearly continuous, in fact $\psi_\Lambda$  is  locally constant   and orbitwise injective, i.e. for any $z\neq z'\in \Lambda$ there is $k\in \mathbb{Z}$ with $\psi_\Lambda(\sigma^kz)\neq \psi_\Lambda(\sigma^kz')$. However $\psi_\Lambda$ is  not equivariant, i.e. $\psi_\Lambda\circ \sigma \neq S\circ\psi_\Lambda $. 
  
  The desired embedding of $(\varprojlim_k\Lambda_k,\sigma)$, where $\sigma$ acts here again as the shift map on each coordinate, into $(Y,S)$
will be obtained as the pointwise limit of $(\phi_k)_k$ where  for all $k$ we let $\phi_k:=\psi_{\Lambda_k}\circ p_k$ with  $p_k:\varprojlim_k\Lambda_k\rightarrow \Lambda_k$ being the projection on $\Lambda_k$. Of course there is a priori no reason that such a sequence is converging pointwisely and its limit is a Borel embedding. In the next paragraph we introduce sufficient conditions on the inverse limit  to ensure these properties. Inverse limits of specifications satisfying these conditions will be said \textbf{admissible}.

\subsubsection{Admissible inverse limit}
An admissible inverse limit of specifications  $\varprojlim_k\Lambda_k$ almost embeds in $(Y,S)$ via the family of continuous  maps $(\psi_{\Lambda_k})_k$, in the sense that the following diagram commutes on a larger and larger set within  a smaller and smaller error 
as $k$ is increasing.
$$  \xymatrix{ \cdots \ \Lambda_k  \ar[rd]_{\psi_{\Lambda_{k}}} && \Lambda_{k+1} \ \cdots \ar[ld]^{\psi_{\Lambda_{k+1}}}  \ar[ll]_{\pi_k}\\ & Y }$$

 To be more precise we first recall the notion of orbit capacity introduced in \cite{sw}. For a topological system $(X,T)$ the \textbf{orbit capacity }of a subset $C$ of $X$ is defined as $$\ocap(C):=\lim_n\frac{1}{n}\sup_{x\in X}\sharp \{0\leq k< n, \ T^kx\in C\}.$$
For a Borel subset $C$ one easily checks by using the pointwise ergodic theorem that the orbit capacity of $C$ is larger than $\mu(C)$ for any $T$-invariant probability measure $\mu$. A Borel subset $C$ of $X$ will be called a \textbf{full set} if $\mu(C)=1$ for all $\mu$.

We define now precisely admissible inverse limits of specifications. 
For any $\epsilon>0$ we let $L_\epsilon$ be the sublinear function involved in the definition (on page 2) of the $\epsilon$-almost specification property of $(Y,S)$. 
 For a given  specification $(\Lambda,\sigma)$ and for all $x=(x^q)_q\in \Lambda$ we let $I_\Lambda(x)$ be the interval of integers $ \llbracket m,...,0,...,n \rrbracket$ such that $B_\Lambda(x):=x^m...x^0...x^n$ is a generating word  of $\Lambda$ (thus $I_\Lambda(x)$ is the interval of coordinates of $B_\Lambda(x)$). We also let $J_\Lambda(x)\supset I_\Lambda(x)$ be the interval of coordinates of the concatenation $ \tilde{B}_\Lambda(x)$ of $B_\Lambda(x)$ with the next generating word of $\Lambda$ used to form $x$ (which is just $B_{\Lambda}\left(\sigma^{n+1}(x)\right)$ with the previous notations, so that we have $ \tilde{B}_\Lambda(x)=B_\Lambda(x)B_{\Lambda}(\sigma^{n+1}(x))$).  Finally for a finite subset $I$ of $\mathbb{Z}$  we let $d_I$ be the distance on $Y$ given by $d_I(x,y):=\max_{k\in I}d(S^kx,S^ky)$ for $x,y\in Y$.

 Let $(\epsilon_k)_k$ be a fixed decreasing sequence of positive real numbers  with $\sum_{k>0}\epsilon_k<\epsilon_0/2$ and let $(\epsilon'_k)_k$ be the sequence given for all $k\geq 1$ by  $\epsilon'_k:=\epsilon_0-\sum_{l=1}^k\epsilon_k>\epsilon_0/2>\epsilon_{k+1}$ (we let $\epsilon'_0:=\epsilon_0$).   For a fixed $k$ the specification $\Lambda_k$ will be a $ L_{\epsilon_{k+1}}$-specification. We recall that by the almost weak specification property the system $(Y,S)$ satisfies the $(\epsilon_k, L_{\epsilon_k})$ specification property for all $k$). Moreover the specification $\Lambda_k$ is $\epsilon'_k$-admissible. We will denote by $(E_k,o_k, N_k,m_k,l_k)$ the datas associated to the specification $\Lambda_k$.

\begin{defi}\label{admiss}  With the previous notation an inverse limit $\varprojlim_k\Lambda_k$ of specifications with factor maps $(\pi_k)_k$ will be said to be \textbf{admissible} whenever there exists a sequence $(P_k)_k$ of positive integers with the following properties  for all $k\geq 0$:
\begin{enumerate}
\item  $\Lambda_k$  is a $ L_{\epsilon_{k+1}}$-specification which is $\epsilon'_k$-admissible,
\item $\pi_k:\Lambda_{k+1}\rightarrow \Lambda_k$ is a block code, i.e.
the $\pi_k$-image  of every generating word in $\Lambda_{k+1}$ is a word with the same length given by the finite concatenation of generating words in $\Lambda_k$, 
\item $(\Lambda_k,\sigma)$ is topologically conjugated to $(\Lambda_0\times Z_{\underline{n}}^k,\sigma)$ via a conjugacy map 
$\theta_k:\Lambda_k\rightarrow \Lambda_0\times Z_{\underline{n}}^k$ satisfying $\theta_k\circ \pi_k=(\Id_{\Lambda_0}\times\eta_{k})\circ \theta_{k+1}$ with $\Id_{\Lambda_0}$ being the indentity map on $\Lambda_0$,
\item for $x=(x_{k'})_{k'}\in  \varprojlim_{k'} \Lambda_{k'}$ and $(\phi_k)_k$ as in Subsection \ref{dff},  $$[d(\phi_{k+1}(x),\phi_k(x))\geq\epsilon_k]\Rightarrow [\exists l\in \mathbb{Z} \ \sigma^l(x_{k+1})=o_{k+1} \text{ with }|l|\leq P_{k+1}],$$
\item $$P_k\geq m_k+L(m_k)+l_k+\overline{N_{k-1}}$$ and $$ P_k/\underline{N_k}<\frac{1}{ 2^{k+1}}.$$
  \end{enumerate}
\end{defi}
From the third item it follows that an admissible inverse limit $\varprojlim_k\Lambda_k$ is topologically conjugated to $\Lambda_0\times Z_{\underline{n}}$ through the conjugacy map $(x_k)_k\in \varprojlim_k\Lambda_k\mapsto \left(\theta_k(x_k)\right)_k$. 

We consider for $k\in \mathbb{N}$ the following subsets $F_k$ and $G_k$ of  $ \varprojlim_{k'} \Lambda_{k'}$:  $$F_k:=\{x=(x_{k'})_{k'}\in  \varprojlim_{k'} \Lambda_{k'}, \ \exists l\in \mathbb{Z} \ \sigma^l(x_{k})=o_{k} \text{ with }|l|\leq P_k\}$$
and $$G_k:=\{x=(x_{k'})_{k'}\in  \varprojlim_{k'} \Lambda_{k'}, \ \exists l\in \mathbb{Z} \ \sigma^l(x_{k})=o_{k} \text{ with }|l|\leq 3 P_k\}.$$

 According to the fourth item in the above definition we have $d(\phi_{k+1}(x),\phi_k(x))<\epsilon_k$ for $x\notin F_{k+1}$. It follows also from the inequality  $P_k\geq m_k+L(m_k)+l_k$ that for $x=(x_{k'})_{k'}\notin F_k$  the point $x_k\in \Lambda_k$ belongs to the piece of orbit of some point in $E_k$ (with $E_k$ being the set associated to the specification $\Lambda_k$).  Thus we have $\phi_k\circ \sigma (x)= S\circ \phi_k(x)$. 
If $x\notin G_{k+1}$ then we have $\sigma^j(x)\notin F_{k+1}$ for any $j\in J_{\Lambda_k}(x_k)$ because $P_{k+1}$ is larger than the length of any generating word of $\Lambda_k$. Therefore $d_{J_{\Lambda_k}(x_k)}(\phi_{k+1}(x),\phi_k(x))<\epsilon_k$.  We observe finally that 
$$\ocap\left(G_k\right)\leq \frac{6P_k}{\underline{N_k}}<1/2^k.$$

In the next lemma we show by a Borel-Cantelli argument that the sequence $(\phi_k)_k$ is converging on a full set. The maps $(\phi_k)_k$ are more and more equivariant as they sends finite pieces of orbits of the inverse limit with longer and longer length to pieces of orbit of $(Y,S)$ with the same length. It follows that   the limit map is equivariant. Morevover  it is injective as a consequence of the  $\epsilon'_k$-admissibility of the  $ L_{\epsilon_{k+1}}$-specification $\Lambda_k$.

\begin{lemma}\label{rat}
Let $(\varprojlim_k\Lambda_k,\sigma)$ be an admissible inverse limit of specifications. Then the continuous maps $(\phi_k)_k$, as defined in Subsection \ref{dff}, are converging  on a full set $G$  to an injective  function $\phi$ satisfying  $\phi\circ T=S\circ \phi$.
\end{lemma}

The map $\phi$ may be then extended equivariantly  on the whole set $\varprojlim_k\Lambda_k$ (see Remark 14 of \cite{Bud} for details).

\begin{proof}[Proof of Lemma \ref{rat}]
By  Borel-Cantelli Lemma we get $\mu(\limsup_kG_k)=0$ for any $\sigma$-invariant probability measure $\mu$. We let $G$ be the full set  given by the complement of $\limsup_kG_k$. For $x\in G$ we have $d(\phi_k(x),\phi_{k+1}(x))< \epsilon_k$  and $\phi_k\circ \sigma (x)= S\circ \phi_k(x)$ when $k$ is large enough. Consequently the sequence $(\phi_k)_k$ is converging pointwisely on $G$ to a limit  $\phi$ satisfying $\phi\circ T=S\circ \phi$. It remains to show $\phi$ is injective on $G$. For $z=(z_k)_k$ in  $G$ the sequence of  intervals  $\left(I_{\Lambda_k}(z_k)\right)_k$ is going to the whole set $\mathbb{Z}$ of integers, in other terms any bounded interval $K$ is contained in $I_{\Lambda_k}(z_k)$ for large enough $k$.  Moreover 
as $\pi_k$ is a block code, the word $B_{\Lambda_k}(z_k)$ completely determines $z_i$ for $k\geq i$ on 
$I_{\Lambda_k}(z_k)$. Thus for $x=(x_k)_k$ and $y=(y_k)_k$ two distinct points in $G$ there exists an integer $i>1$ such that $B_{\Lambda_k}(x_k)\neq B_{\Lambda_k}(y_k)$ for $k\geq i$. By taking $i$ 
large enough we may also assume  that for $k\geq i$ we have $d_{J_{\Lambda_k}(x_k)}(\phi_k(x ),\phi_{k+1}
(x))<\epsilon_k$ and $d_{J_{\Lambda_k}( y_k)}(\phi_k(y),\phi_{k+1}
(y))<\epsilon_k$. In particular $\phi(x )=\lim_k\phi_k(x)\in \Delta_{\epsilon''_i}\left(\tilde B_{\Lambda_i}(x_i)\right) $ and $\phi(y )=\lim_k\phi_k(y)\in 
\Delta_{\epsilon''_i}\left(\tilde B_{\Lambda_i}(y_i)\right) $ with $\epsilon''_i:=\sum_{i<k}\epsilon_k\leq \epsilon'_i$. By definition of $B_\Lambda$ the interval of coordinates of $B_{\Lambda_i}(x_i)$ and $B_{\Lambda_i}(y_i)$ both contain $0$. As $\Lambda_i$ is 
$\epsilon'_i$-admissible it follows from Lemma \ref{ent} that $\Delta_{\epsilon'_i}\left(\tilde B_{\Lambda_i}(x_i)\right)\cap 
\Delta_{\epsilon'_i}\left(\tilde B_{\Lambda_i}(y_i)\right)  =\emptyset$ since $B_{\Lambda_i}(x_i)$ and $B_{\Lambda_i}(y_i)$ are 
distinct.  Therefore we conclude $\phi(x)\neq \phi(y)$. \end{proof}

\subsection{Construction of admissible inverse limit of specifications.}
We consider a system $(Y,S)$ with the almost specification property.

\begin{prop}\label{prp}
Let $0<\alpha<\htop(S)$. There is an admissible  inverse limit  of specifications 
$\varprojlim_k\Lambda_k$ with $\htop(\Lambda_0)$ larger than $\alpha$.
\end{prop}

\begin{proof}We recall  that  $(E_k,o_k, N_k,m_k,l_k)$ denote  the datas associated to the specification $\Lambda_k$.
 We first consider an $\epsilon_0$-admissible simple $L_{\epsilon_1}$-specification $\Lambda'_0$ with topological entropy larger than $\alpha$ as built in Proposition \ref{pp}. We let $n_0,n_0+1$ be the length of the generating words in $\Lambda'_0$. In Proposition \ref{pp} we  may choose the  set $E'_0$ of the specification  $\Lambda'_0$ having  the form $E'_0=E_0\coprod E''_0\coprod \{a_0,b_0\}$ (where $\coprod$ denotes the disjoint union of sets) 
  with $\sharp E''_0 \geq e^{\alpha''n_0}$ and $\sharp E_0\geq e^{\alpha'n_0}$ for 
  $\alpha''>\alpha'>\alpha$. Indeed one only needs to take $E'_0$ with $\sharp E'_0\geq e^{\alpha''n_0}+e^{\alpha'n_0}+2$. We may also assume the  specifications $\Lambda''_0,\Lambda_0\subset \Lambda'_0$ associated to $E''_0, E_0$ are also simple, therefore topologically mixing. Finally for large enough $n_0$ 
 we have  $\htop(\Lambda''_0)>\htop(\Lambda_0)>\alpha$ according to Lemma \ref{entr}.    In particular any inverse limit with $\Lambda_0$ as a first term has topological entropy larger than $\alpha$. We let $n_0,n_0+1$ be the length of the generating words in $\Lambda'_0$, $\Lambda''_0$ and $\Lambda_0$.


By induction we build then for all $k$  a $\epsilon'_k$-admissible $L_{\epsilon_k}$-specification $\Lambda'_k$  
such that the associated set $E'_k$ may be written $E'_{k}=E_k\coprod E''_k\coprod \{a_k,b_k\}$
 and the specifications $\Lambda''_k, \Lambda_k\subset \Lambda'_k$ associated to  $E''_k, E_k$ satisfy  $\htop(\Lambda''_{k-1})\geq \htop (\Lambda''_k)>\htop(\Lambda_k)\geq\htop(\Lambda_{k-1})$. The sequence of specifications $(\Lambda_k)_k$ associated to $E_k$ will define an admissible  inverse limit.

 The points $a_k, b_k\in E_k$ will be used to define the marker point $o_{k+1}$ of the specification $\Lambda_{k+1}$. Let us explain shortly the role of $\Lambda''_k$ in the construction : to get the specification $\Lambda_k$ we will need to compress the information given by $\Lambda_{k-1}$. As the topological entropy of $\Lambda''_{k-1}$ is larger than  that  of $\Lambda_{k-1}$ it can be achieved by building  generating words in $\Lambda_k$ from words in $\Lambda''_{k-1}$. 

 We assume $\Lambda_{i-1}$ and $\pi_{i-2}$   already built for $i\leq k$. In a given specification,  a concatenation of $n$-generating words is called a $n$-block. Generating words in $\Lambda_k$ are obtained from every $n_k$-blocks and $(n_{k}+1)$-blocks in $\Lambda_{k-1}$ after a series of modifications, where $(n_k)_k$ is a fast increasing sequence. In particular the generating words in $\Lambda_k$ have length in the interval of integers $\llbracket \prod_{i=0}^kn_i,\prod_{i=0}^k(n_i+1)\rrbracket$ and all values are achieved (it implies $\Lambda_k$ is topologically mixing). The  conjugacy map $\theta_k:\Lambda_k\rightarrow \Lambda_0\times Z_{\underline{n}}^k$ will be obtained for its first coordinate  by decompressing the information from $\Lambda_k$ to $\Lambda_0$, whereas the  second coordinate gives the structure of blocks in $\Lambda_l$ for $l\leq k$. 

\subsubsection{Marker word in $\Lambda_k$.}

We let $A=A_{k-1}$ and $B=B_{k-1}$ be the generating word in $\Lambda'_{k-1}$ associated to $a_{k-1}$ and $b_{k-1}$. We consider a block $C_k$ in $\Lambda'_{k-1}$ of the form 
$C_k\in \{A^qB^q, \ q\in \mathbb{N}^*\}$ with $A^q=\underbrace{A\cdots A}_{q \text{ times} }$.

  If $\Delta^{\Lambda_{k-1}}_{\epsilon'_{k-1}}(C_k)\cap \Delta^{\Lambda_{k-1}}_{\epsilon'_{k-1}}(\sigma^lC_k)\neq \emptyset$ for some integer $l$ the blocks  $\sigma^lC_k$ and $ C_k$ should  coincide on the intersection of their interval of coordinates (which may be empty) by  $\epsilon'_{k-1}$-admissibility of $\Lambda'_{k-1}$ according to Lemma \ref{ent}.   As $A$ and $B$ are distinct this intersection is  non empty only for $l=0$.   We  put $m_k$ equal to the length of $C_k$ and we pick $o_k\in \Delta^{\Lambda_{k-1}}_{\epsilon_k}(C_k)$. For $0< l\leq \frac{3}{4}m_k$ the intervals of coordinates of $C_k$ and $\sigma^lC_k$ are not disjoint, therefore $\Delta^{\Lambda_{k-1}}_{\epsilon'_{k-1}}(C_k)\cap \Delta^{\Lambda_{k-1}}_{\epsilon'_{k-1}}(\sigma^lC_k)=\emptyset$. Thus  we have 
   \begin{align*}
   B(o_k,m_k,\epsilon'_k)\cap S^lB(o_k,m_k,\epsilon'_k)&\subset \Delta^{\Lambda_{k-1}}_{\epsilon'_{k-1}}(C_k)\cap \Delta^{\Lambda_{k-1}}_{\epsilon'_{k-1}}(\sigma^lC_k)=\emptyset.
   \end{align*}
    The point $o_k$ will play the role of the marker in the specifications $\Lambda_k, \Lambda'_k$ and $\Lambda''_k$.

\subsubsection{Generating words in $\Lambda_k$, $\Lambda'_k$ and $\Lambda''_k$.}
For a $n_k$- or $(n_k+1)$-block $D_k$ in $\Lambda_{k-1}$ or $\Lambda''_{k-1} $ we replace  the first letters by $o_k,...,S^{m_k}o_k,*^{L_{\epsilon_k}(m_k)}$ and the last letters by $*^{l_k}$ with $l_k=L_{\epsilon_k}\left(\prod_{l=0}^k(n_l+1)\right)$. We consider the subword $D'_k$  given by the remaining middle part of $D_k$. 

We associate to the $n_k$- or $(n_k+1)$-block $D_k$  a word $D''_k$ in $\Lambda'_{k-1}$ with the same length as 
$D'_k$ (see the next Subsection \ref{reco}). Then we will pick up $x_k\in \Delta^{\Lambda_{k-1}}_{\epsilon_k}(D''_k)$ and we replace the subword $D'_k$ 
of $D_k$ by $(x_k, Sx_k,...,S^{|D'_k|-1}x_k)$. The set  $E_k$ (resp. $E''_k\coprod \{a_k,b_k\}$)  is  defined as the  set of all points $x_k$ obtained in this way from all $n_k$- and $(n_{k}+1)$-blocks $D_k$ in $
\Lambda_{k-1}$ (resp. in $\Lambda''_{k-1}$). 
  We let $\Lambda_k$ (resp. $\Lambda''_k$) be the 
specification with respect to $\left(E_k \text{ (resp. }E''_k\text{)},   o_k, N_k, m_k,l_k \right)$. The map $D_k\mapsto 
D''_k$, defined below, will be injective on $n_k$- and $(n_{k}+1)$-blocks of $\Lambda_{k-1}$ so that we will obtain a bijective map $\pi_{k-1}$ from the 
$1$-blocks of the specification $\Lambda_k$ to the $n_k$- and $(n_{k}+1)$-blocks of $\Lambda_{k-1}$.   This map induces the required topological conjugacy $\theta_k$  between $\Lambda_{k}$ and $\Lambda_{0}\times Z_{\underline{n}}^k$ (define $\theta_k$ such that we have  $\theta_{k-1}\circ \pi_{k-1}=(\Id_{\Lambda_0}\times\eta_{k-1})\circ \theta_{k}$ and the $n^{th}$-term of the second component of $\theta_k$ in $\{0,1,...,k\}^\mathbb{Z}$ is equal to $k$ if and only if $n$ is the first coordinate of a $1$-block in $\Lambda_k$).

Moreover 
when $\prod_{i=0}^kn_i \gg m_k \gg \prod_{i=0}^{k-1}(n_i+1) $ we will get $\htop(\Lambda_{k-1})\leq  \htop(\Lambda_k)<\htop(\Lambda''_k)
\sim \htop(\Lambda''_{k-1})$.

\subsubsection{Recoding and factor maps $\pi_k$.}\label{reco}
When $D_k$ is a $n_k$- or $(n_k+1)$-block in $\Lambda''_{k-1}$ we just let $D''_k=D'_k$ (there is no compression of information in this case).  We now examine the case when $D_k$ is a $n_k$- or $(n_{k}+1)$-block in $\Lambda_k$. For a specification $\Lambda$ and an integer $p$ we let $\mathcal{W}^\Lambda(p)$ be the set of words of $\Lambda$ of length $p$ and $\mathcal{V}^\Lambda(p)$ be the subset of $\mathcal{W}^\Lambda(p)$ given by words which ends with a generating word. A word  $V\in \mathcal{V}^\Lambda(p)$ may be written (uniquely) as the concatenation of a prefix, which contains no generating words, with a (complete) block $s(V)$ of $\Lambda$, called the suffix of $V$. 

As the topological entropy of $\Lambda''_{k-1}$ is strictly larger than the topological entropy of $\Lambda_{k-1}$ there is an integer $p_k$ with $\prod_{i=0}^kn_i\gg p_k\gg \max(l_k,m_k)$ such that
$$\sharp  \mathcal{V}^{\Lambda''_{k-1}}
(p_k) \geq \sharp  \left(\mathcal{W}^{\Lambda_{k-1}}
\left(m_k+L(m_k)+p_k\right)\times \mathcal{W}^{\Lambda_{k-1}}(l_k)\right).$$ In fact we may replace in the above inequality the set $ \mathcal{V}^{\Lambda''_{k-1}}(p_k)$ by a subset $\mathcal{U}^{\Lambda''_{k-1}}(p_k)\subset  \mathcal{V}^{\Lambda''_{k-1}}(p_k)$ such that any two distinct elements  in this subset have distinct suffixes (indeed the length of the prefixes is of order $\prod_{i=0}^{k-1}n_{i}$ and we choose  $p_k\gg m_k\gg \prod_{i=0}^{k-1}n_{i}$).
We may thus consider a  surjective map $\Theta_k $ from $\mathcal{U}^{\Lambda''_{k-1}}(p_k)$ onto the product set $\mathcal{W}^{\Lambda_{k-1}}\left(m_k+L(m_k)+p_k\right)\times \mathcal{W}^{\Lambda_{k-1}}(l_k)$. 
As explained above  we replace  the first letters of $D_k$ by $o_k,...,S^{m_k}o_k,*^{L_{\epsilon_k}(m_k)}$ and the last letters by $*^{l_k}$ with $l_k=L_{\epsilon_k}\left(\prod_{l=0}^k(n_l+1)\right)$. 
We write $D_k$ as the concatenation $D_k=D_k^1D_k^2D_k^3D_k^4$  with $D'_k=D_k^2D_k^3$ and $|D_k^2|=p_k$. To get $D''_k$ we only modify $D_k^2$. We re-encode the information lost in $D_k^1$, $D_k^2$ and  $D_k^4$ in a word of length $p_k$ via $\Theta_k$. More precisely we let   
 $D''_k$ be the concatenation $\Theta_k^{-1}(D_k^1D_k^2,D_k^4) D_k^3$  for some fixed right inverse  $\Theta_k^{-1}$ of $\Theta_k$. Recall we take then $x_k\in \Delta^{\Lambda_{k-1}}_{\epsilon_k}(D''_k)$ and we replace the subword $D'_k$ 
of $D_k$ by $(x_k, Sx_k,...,S^{|D'_k|-1}x_k)$ to get a generating word of $\Lambda_k$. Any  coordinate in $D_k$ lying in $D_k^3$ and the corresponding coordinate in the generating word of $\Lambda_k$ are $\epsilon_k$-close. In particular the second last item of Definition \ref{admiss}  is satisfied for $P_k:=\max\left(m_k+L(m_k)+p_k, l_k\right)$. Finally the condition on $P_k$, in the last item of the same definition, holds true  for $\prod_{i=0}^kn_i\gg p_k\gg \max(l_k,m_k)$.

The $1$-blocks of the specification $\Lambda_k$ are in bijection  with the $n_k$- and $(n_k+1)$-blocks of $\Lambda_{k-1}$.   The (surjective) factor map $\pi_{k-1}:\Lambda_k\rightarrow \Lambda_{k-1}$ is the associated block code.

\subsubsection{Admissibility of $\Lambda'_k$.}
We prove now by induction on $k$ that $\Lambda'_k$ (and therefore $\Lambda_k$) is $\epsilon'_k$-admissible.  We first check $E'_k=E_k\coprod E''_k\coprod \{a_k, b_k\}$ is $(N,\epsilon'_k)$-separated. 
 Let $x\neq y\in E'_k$ with $N_k(x)=N_k(y)$. The points $x$ and $y$ belongs respectively 
to $\Delta^{\Lambda_{k-1}}_{\epsilon_k}(D''_k)$ and $\Delta^{\Lambda_{k-1}}_{\epsilon_k}(\tilde D''_k)$ with $D''_k$ and $
\tilde{D}''_k$ being words on the same interval of coordinates with distinct suffixes. By  $\epsilon'_{k-1}$-admissibility of $\Lambda'_{k-1}$ the sets   $\Delta_{\epsilon'_{k-1}}(D'_k)$ and $ \Delta_{\epsilon'_{k-1}}(\tilde D'_k)$ are disjoint according to Lemma \ref{ent} and thus it is also the case for  $B(x,N_k(x),\epsilon'_k)$ and $B(y,N_k(y),\epsilon'_k)$.

The property required on the marker point having already being checked, it is enough to see 
$B(o_{k},m_k,\epsilon'_k)\cap S^iB(x,N_k(x),\epsilon'_k)=\emptyset$ for any $-\frac{2}{3}m_k\leq i\leq N_k(x)-\frac{2}{3}m_k$ and for any $x\in E'_k$. But this follows as above from the $\epsilon'_{k-1}$-admissibility of $\Lambda'_{k-1}$ and the fact that $o_k$ and elements of $E'_k$ are issue from disjoint families  of $1$-blocks in $\Lambda'_{k-1}$ (associated to $\{a_{k-1}, b_{k-1}\}$ for $o_k$ and  to $E_{k-1}\coprod E''_{k-1}$ for elements in $E'_k$). This concludes the proof of Proposition \ref{prp}.
\end{proof}

\subsection{Almost Borel universality for systems with the almost weak specification}
\begin{proof}[Proof of Theorem \ref{fin}]
Let $(Y,S)$ be a topological system with the almost weak specification property. By Proposition \ref{prp} there is an admissible inverse limit $(\varprojlim_k\Lambda_k,\sigma)$ of specifications of $(Y,S)$ with $\htop(\Lambda_0)$ arbitrarily close to $\htop(S)$.  By Lemma \ref{fest} such an inverse limit  is almost Borel $\{\hbor<\htop(\Lambda_0)\}$-universal. Finally  by Lemma \ref{rat}   
this inverse limit  almost Borel embeds  in $(Y,S)$ so that $(Y,S)$ is also almost Borel  $\{\hbor<\htop(\Lambda_0)\}$-universal. As it holds for $\htop(\Lambda_0)$ arbitrarily close to $\htop(S)$ the system is $\{\hbor<\htop(S)\}$-universal.
\end{proof}

\section{Fully topologically and almost Borel  universal systems}
We discuss now the relevance of fully embedding in our context. One can not embed topologically an e.a.z.  systems $(X,T)$  in a system $(Y,S) $ with the weak specification property  and larger topological entropy in  such a way the invariant measures of the embedded system $(X', S)$ have full support. Indeed $Y\setminus X'$ being open, it should be the empty set and  therefore $(X,T)$ is topologically conjugated to $(Y,S)$. In particular we have the equality  $\htop(X,T)=\htop(Y,S)$, which contradicts our assumption. \\

In the almost Borel setting,  the concept of fully embedding makes sense and an easy adaptation of the construction in the above section give the following statement. A system $(Y,S)$ is said to be fully almost Borel  $\{\hbor<\htop(S)\}$-universal when any Borel system $(X,T)$ with $\hbor(X,T)<\htop(Y,S)$ almost Borel embeds  in $(Y,S)$ such that the pushforward of $T$-invariant measures by the embedding  have full support.

\begin{theo}\label{lat}
Any topological system $(Y,S)$ with the almost weak specification property is fully  almost Borel  $\{\hbor<\htop(S)\}$-universal.
\end{theo}

\begin{proof}
A topological system $(Y,S)$ with the almost weak specification property is topologically transitive. 
Let $z$ be a point  in $Y$ with a dense forward $S$-orbit. For all $k$ we fix an integer $r_k$ such that 
$Z,Sz,...,S^{r_k-1}z$ is $\epsilon_k$-dense, i.e. this finite piece of orbit meets all balls of radius $\epsilon_k$. Then in the construction of the specifications $(\Lambda_k)_k$ we replace the final subword $*^{l_k}$ of the generating words by $*^{l_k},z,Sz,...,S^{r_k-1}z,*^{L_{\epsilon_k}(r_k)}$. The construction remains valid provided one choose the length of the marker word $m_k\gg r_k$. Finally for any ergodic $T$-invariant measure $\mu$ the induced measure $\phi_k^*\mu$ on $Y$ give positive weight to any ball of radius $\epsilon_k$ and by admissibility of the inverse limit $\varprojlim_k\Lambda_k$  so does the measure $\phi^*\mu$ for any ball of radius $\epsilon_k+\epsilon''_k=\epsilon''_{k-1}$. As $k$ may be chosen arbitrarily we conclude that $\phi^*\mu$ has full support. 
\end{proof}
\end{large}

\end{document}